\documentclass[reqno,12pt]{amsart}
\usepackage{amssymb,amsmath,amsthm,amsxtra,calc,bm, color}
\usepackage{enumerate}
\usepackage[margin=.95in]{geometry}

\usepackage[scr=boondoxo]{mathalfa}
\usepackage{mathtools}

\usepackage{mleftright}
\mleftright

\usepackage{nccmath}
\usepackage{cases}

\usepackage{calc}
\usepackage{graphicx}

\usepackage{hyperref}

\usepackage[final]{microtype}

\def\Z{\mathbb{Z}}
\def\Q{\mathbb{Q}}
\def\R{\mathbb{R}}

\def\C{\mathbb{C}}

\DeclareMathOperator{\im}{Im}
\DeclareMathOperator{\re}{Re}

\def\SL{{\rm SL}}
\def\PSL{{\rm PSL}}

\newcommand{\pfrac}[2]{\left(\frac{#1}{#2}\right)}

\renewcommand{\pmatrix}[4]{\left(\begin{smallmatrix}#1 & #2 \\ #3 & #4\end{smallmatrix}\right)}
\renewcommand{\bar}[1]{\overline{#1}}

\DeclareMathOperator{\sgn}{sgn}

\def\ep{\varepsilon}

\newtheorem{theorem}{Theorem}[section]
\newtheorem{lemma}[theorem]{Lemma}

\newtheorem{proposition}[theorem]{Proposition}

\theoremstyle{remark}

\numberwithin{equation}{section}

\mathtoolsset{showonlyrefs}

\DeclareMathOperator{\Res}{Res}

\title{The Kohnen-Zagier formula for Maass forms for $\Gamma_0(4)$}

\date{\today}

\author{Nickolas Andersen}
\email{nick@math.byu.edu}
\address{Brigham Young University, Provo, UT 84602}

\begin{document}

\begin{abstract}
	We extend a formula of Duke, Imam\=oglu, and T\'oth (which itself is a generalization of the Katok-Sarnak formula) to prove the Kohnen-Zagier formula for Maass forms for $\Gamma_0(4)$.
\end{abstract}

\maketitle

\allowdisplaybreaks


\section{Introduction}

Let $d$ be a fundamental discriminant.
The Kohnen--Zagier formula \cite{KZ} relates the $|d|$-th coefficient of a holomorphic Hecke eigenform $g$ of half-integral weight on $\Gamma_0(4)$ to $L(\frac 12,f\times \pfrac d\cdot)$, where $f$ is the Shimura lift of $g$.
The formula is an explicit version of the general relation of Waldspurger \cite{walds}.
Here we show how the ideas of \cite{DIT-geom} and \cite{and-inf} can be combined to give a short proof of the Kohnen--Zagier formula for Maass cusp forms for $\Gamma_0(4)$.
We adopt the notation of \cite{DIT-geom}; see the next section for details.
\begin{theorem} \label{thm:main}
Let $\varphi$ be an even Hecke--Maass cusp form of weight $0$ for $\SL_2(\Z)$ with Fourier expansion
\begin{equation}
	\varphi(z) = 2\sqrt y \sum_{n\neq 0} a(n) K_{ir}(2\pi|n|y)e(nx).
\end{equation}
Then there exists a unique Maass cusp form $\psi$ of weight $1/2$ for $\Gamma_0(4)$ with Fourier expansion
\begin{equation}
	\psi(z) = \sum_{0\neq n\equiv 0,1(4)} b(n) W_{\frac 14\sgn(n),\frac{ir}{2}}(4\pi|n|y)e(nx)
\end{equation}
such that for any fundamental discriminant $d\equiv 0,1\pmod{4}$ we have
	\begin{equation} \label{eq:kohnen-zagier-formula}
	12\pi|d||b(d)|^2 = \langle \varphi,\varphi \rangle^{-1} \left|\Gamma(\tfrac 12 - \tfrac{\sgn d}{4}+\tfrac{ir}{2})\right|^2 L(\tfrac 12,\varphi\times \chi_d).
	\end{equation}
Here $\chi_d = \pfrac{d}{\cdot}$ and $L(s,\varphi\times\chi_d)$ denotes the analytic continuation of the $L$-function
\begin{equation}
	L(s,\varphi\times\chi_d) = \sum_{n=1}^\infty \frac{a(n)\chi_d(n)}{n^s}.
\end{equation}
\end{theorem}

The $d=1$ case of Theorem~\ref{thm:main} is a corollary of the main result of Katok and Sarnak \cite{KS}, which relates the product $b(d)\bar b(1)$ to one of the quantities
\begin{equation}
	\sum_{Q\in \Gamma\backslash \mathcal Q_d} \varphi(z_Q) \qquad \text{ or } \qquad \sum_{Q\in \Gamma\backslash \mathcal Q_d} \int_{C_Q} \varphi(z) y^{-1} \, |dz|,
\end{equation}
depending on whether $d$ is negative or positive, respectively.
Here $\mathcal Q_d$ is a set of integral binary quadratic forms of discriminant $d$, $\Gamma = \PSL_2(\Z)$, $z_Q$ is the root of $Q(z,1)$ in the complex upper half-plane $\mathcal H$, and $C_Q$ is a hyperbolic geodesic, finite if $d>1$ and infinite if $d=1$ (see Section~\ref{sec:background} for details).
In the case $d=1$ there is one term in the sum, namely $Q=[0,1,0]$, and their formula reads
\begin{equation}
	12\sqrt\pi  |b(1)|^2 =\langle \varphi,\varphi \rangle^{-1} \int_0^\infty \varphi(iy) y^{-1} \, dy.
\end{equation}
The latter integral evaluates to a multiple of $L(\frac 12,\varphi)$.
To prove their formula, Katok and Sarnak modify the theta lift of Shintani \cite{shintani} and Niwa \cite{niwa}.
With some extra work, their method can probably produce a formula for $|b(d)|^2$ where $d$ is any positive fundamental discriminant.

Using the Kuznetsov trace formula and some ideas from the proof of the Selberg trace formula, Bir\'o \cite{biro} extended the Katok--Sarnak formula to general level for a pair of positive discriminants $d$ and $d'$ such that $d$ is fundamental.
The main result of \cite{biro} is a relation between $b(d)\bar b(d')$ and the twisted sums
\begin{equation}
	\sum_{Q\in \Gamma\backslash\mathcal Q_{dd'}} \chi_d(Q) \int_{C_Q} \varphi(z) y^{-1} |dz|,
\end{equation}
where $\chi_d$ is a character of the finite group $\Gamma\backslash\mathcal Q_{dd'}$.
When $d=d'$ the sum above evaluates to a multiple of $L(\frac 12,\varphi\times\chi_d)$.
It is not clear whether the methods of \cite{KS} or \cite{biro} can be extended to cover the case where $d,d'$ are negative.

In \cite{DIT-geom}, Duke, Imamo\=glu, and T\'oth generalized the formulas of Katok--Sarnak and B\'iro in the case of Maass forms for $\Gamma_0(4)$ to allow for two discriminants $d,d'$ of any sign, as long as $d$ is fundamental.
In the new case, when $d$ and $d'$ are both negative (and $dd'$ is not a square), 
Theorem~4 of \cite{DIT-geom} gives a relation between $b(d)\bar b(d')$ and
\begin{equation}
	\sum_{Q\in \Gamma\backslash\mathcal Q_{dd'}} \chi_d(Q) \int_{\mathcal F_Q} \varphi(z) \, \frac{dxdy}{y^2},
\end{equation}
where $\mathcal F_Q$ is a finite area hyperbolic surface with boundary $C_Q$.
The case $d=d'$ is not covered in that theorem because the proof relies on being able to compute the integral over $C_Q$ of a certain Poincar\'e series, and the corresponding integral when $dd'$ is a square does not converge.
Here we use the main idea of \cite{and-inf} to modify the Poincar\'e series in the case $d=d'$ and give a short proof of Theorem~\ref{thm:main}.

The generalization of Theorem~\ref{thm:main} to Maass forms for $\Gamma_0(4N)$, with $N$ odd and squarefree, was proved by Baruch and Mao in \cite{bm}.
Their proof utilizes the powerful tools of automorphic representation theory.

\section{Background} \label{sec:background}

Throughout this paper we make use of several special functions, especially the Bessel functions $I_\nu(x)$, $J_{\nu}(x)$, and $K_{\nu}(x)$, and the Whittaker functions $M_{\mu,\nu}(x)$ and $W_{\mu,\nu}(x)$.
Definitions and properties of these functions can be found in Sections~10 and 13 of \cite{dlmf}.
In the rest of this section, we give some background information on the objects in the introduction, including some standard facts we will need for the proof of the main theorem.
We are mostly following the notation and setup of \cite{DIT-geom}.
Other standard references are \cite{fay,hejhal,roelcke}.

\subsubsection*{Maass cusp forms of weight $0$}
Let $\Gamma=\PSL_2(\Z)$ and let
\begin{equation}
	\Delta_k = y^2\left(\partial_x^2+\partial_y^2\right) - iky\partial_x
\end{equation}
denote the weight $k$ hyperbolic Laplacian.
A function $\varphi:\mathcal H\to\C$ is a Maass form of weight $0$ for $\Gamma$ if it is $\Gamma$-invariant and is an eigenfunction of $\Delta_0$ with eigenvalue normalized by $(\Delta_0+\lambda)\varphi=0$ and
\begin{equation}
	\lambda = \tfrac 14+r^2 \qquad \text{ with }r\geq 0.
\end{equation}
The quantity $r$ is called the spectral parameter of $\varphi$.
We say that $\varphi$ is a Maass cusp form if the constant term in its Fourier expansion is zero, i.e.
\begin{equation}
	\varphi(z) = 2\sqrt y \sum_{n\neq 0} a_\varphi(n) K_{ir}(2\pi|n|y) e(nx)
\end{equation}
for some coefficients $a_\varphi(n)\in \C$.
For each $r\geq 0$ let $\mathcal U_r$ denote the vector space of Maass cusp forms of weight $0$ with spectral parameter $r$.

For each prime $p$, the Hecke operator $T_p$ acts on $\mathcal U_r$ via Fourier expansions as
\begin{equation}
	(T_p \varphi)(z) = 2\sqrt y \sum_{n\neq 0} \left(a_\varphi(pn) + p^{-1}a_\varphi(n/p)\right) K_{ir}(2\pi|n|y)e(nx).
\end{equation}
The Hecke operators commute with each other and with $\Delta_0$, so we can find an orthogonal (with respect to the Petersson inner product $\langle \cdot,\cdot \rangle$) basis $\mathcal B_r$ of $\mathcal U_r$ consisting of Hecke eigenforms.
We will normalize the elements of $\mathcal B_r$, which are called Hecke--Maass cusp forms,
so that $a(1)=1$.
We can also assume that each $\varphi$ is even or odd, meaning that $a(-n)=\pm a(n)$ respectively.

\subsubsection*{Maass cusp forms of weight $\frac 12$}
A function $\psi:\mathcal H\to \C$ is a Maass form of weight $1/2$ for $\Gamma_0(4)$ if it satisfies $\psi(\gamma z) = J(\gamma,z)\psi(z)$ for all $\gamma\in \Gamma_0(4)$, where
\begin{equation}
	J(\gamma,z) = \frac{\theta^*(\gamma z)}{\theta^*(z)}, \qquad \theta^*(z) = y^{1/4} \sum_{n\in \Z}e(n^2z),
\end{equation}
and if $(\Delta_k+\lambda)\psi = 0$ for some $\lambda$.
If $\psi$ is not a constant multiple of $\theta^*$ then $\lambda\geq \frac 14$ and we define the spectral parameter $r$ as before.
Such a $\psi$ is a cusp form if the constant term in its Fourier expansion at each of the cusps of $\Gamma_0(4)\backslash \mathcal H$ is zero.
In this case the Fourier expansion is written
\begin{equation}
	\psi(z) = \sum_{n\neq 0} b_\psi(n) W_{\frac 14\sgn(n), ir}(4\pi |n|y)e(nx).
\end{equation}
Let $\mathcal V_r $ denote the vector space of Maass cusp forms of weight $1/2$ on $\Gamma_0(4)$ with spectral parameter $r/2$.\footnote{Having $\mathcal U_r$ and $\mathcal V_r$ correspond to spectral parameters $r$ and $r/2$, respectively, follows \cite{DIT-geom} and is convenient when working with the Shimura lift, which sends an element of $\mathcal V_r$ into $\mathcal U_r$.}
The Kohnen plus space is the subspace $\mathcal V_r^+$ of $\mathcal V_r$ comprising forms whose Fourier coefficients are supported on indices $n\equiv 0,1\pmod{4}$.
For each prime $p\geq 3$, the Hecke operator $T_{p^2}$ acts on $\mathcal V_r^+$ via Fourier expansions as
\begin{equation}
	(T_{p^2}\psi)(z) = \sum_{0\neq n\equiv 0,1(4)} \left( b_\psi(p^2 n) + \pfrac{n}{p}p^{-1} b_\psi(n) + p^{-1}b_\psi(n/p^2) \right)W_{\frac 14\sgn(n),\frac{ir}{2}}(4\pi|n|y)e(nx).
\end{equation}

\subsubsection*{The Shimura lift} In Theorem~1.2 of \cite{bm}, Baruch and Mao show that for each $\varphi\in \mathcal U_r$ there is a unique $\psi\in \mathcal V_r^+$, spectrally normalized so that $\langle \psi,\psi \rangle = 1$, such that for each prime $p\geq 3$ we have 
\begin{equation} \label{eq:tp^2=a-phi-p}
	T_{p^2} \psi = a_\varphi(p)\psi.
\end{equation}
The form $\varphi$ is called the Shimura lift of $\psi$.
A computation involving \eqref{eq:tp^2=a-phi-p} and $T_p\varphi = a_\varphi (p)\varphi$ shows that
\begin{equation}
	a_\varphi(m) b_\psi(d) = m \sum_{n\mid m}n^{-\frac 32} \pfrac dn b_\psi(m^2 d/n^2)
\end{equation}
for all fundamental discriminants $d$.

\subsubsection*{Quadratic forms and cycles}
For each positive discriminant $D$, let $\mathcal Q_D$ denote the set of (indefinite) integral binary quadratic forms $Q=[a,b,c]$ with $b^2-4ac=D$.
The group $\Gamma$ acts on $\mathcal Q_D$ in the usual way, and the set $\Gamma\backslash \mathcal Q_D$ is finite.
For $Q=[a,b,c]\in \mathcal Q_D$, the equation $ax^2+bxy+cy^2=0$ has two solutions $(x:y)$ in $\mathbb P^1(\R)$.
When $D$ is not a square, each $x/y$ is a real quadratic irrationality, and when $D$ is a square we have either $(x:y)=(1:0)$, corresponding to the point at $i\infty$, or $x/y\in \Q$.
Let $S_Q$ denote the geodesic in $\mathcal H$ connecting the two solutions, and let
$\Gamma_Q\subseteq \Gamma$ denote the isotropy subgroup $\{\gamma\in \Gamma : \gamma Q = Q\}$. 
We follow \cite{DIT-geom} in orienting $S_Q$ clockwise if $a>0$, counterclockwise if $a<0$, and downward if $a=0$ (if $a=0$ the geodesic is a vertical line).%
\footnote{Note that the papers \cite{DIT-cycle,and-inf} give $S_Q$ and $C_Q$ the opposite orientation.}
When $D$ is not a square, $\Gamma_Q$ is infinite cyclic, and when $D$ is a square $\Gamma_Q$ is trivial.
Let $C_Q = \Gamma_Q\backslash S_Q$ be the cycle corresponding to $Q$; it has finite length when $D$ is not a square, and infinite length otherwise.

Let $D=dd'$ be a factorization of $D$ into a fundamental discriminant $d$ and a discriminant $d'$.
The generalized genus character $\chi_d$ associated to the factorization $D=dd'$ is
\begin{equation}
	\chi_d(Q) = 
	\begin{cases}
		\pfrac dn & \text{ if }\gcd(a,b,c,d)=1\text{ and $Q$ represents $n$}, \\
		0 & \text{ if }\gcd(a,b,c,d)>1.
	\end{cases}
\end{equation}
In \cite{GKZ} it is shown that $\chi_d(Q)$ is well-defined on equivalence classes $Q\in \Gamma\backslash \mathcal Q_D$.

It will be helpful to have an explicit description of $\Gamma\backslash \mathcal Q_D$ when $D=d^2$ and $d$ is a fundamental discriminant.
The following is a straightforward generalization of Lemma~3 of \cite{and-inf}.

\begin{lemma} \label{lem:square}
	If $D=d^2$ then the sets
	\begin{equation}
		\{Q=[c,|d|,0] : 0\leq c < |d|\} \qquad \text{ and }\qquad \{Q=[0,|d|,c] : 0\leq c < |d|\}
	\end{equation}
	are both complete sets of representatives for $\Gamma\backslash \mathcal Q_D$.
	In both cases we have
	\begin{equation}
		\chi_d(Q) = \pfrac{d}{c}.
	\end{equation}
\end{lemma}

\section{Proof of Theorem~\ref{thm:main}}
We begin by borrowing a few intermediate results from \cite{DIT-geom}. 
For $\re(s)>1$ let $F_m(z,s)$ denote the Poincar\'e series
\begin{equation}
	F_{m}(z,s) = \sum_{\gamma\in \Gamma_\infty \backslash \Gamma} f_m(\gamma z,s),
\end{equation}
where $f_0(z,s) = y^s$ and for $m\neq 0$
\begin{equation}
	f_m(z,s) = \frac{\Gamma(s)}{2\pi\sqrt{|m|}\Gamma(2s)} M_{0,s-\frac 12}(4\pi|m|y) e(mx).
\end{equation}
The function $F_0(z,s)$ is the usual real analytic Eisenstein series (see \cite[Chapter~15]{IK}) and has Fourier expansion
\begin{equation} \label{eq:eis-fourier}
	F_0(z,s) = y^s + \frac{\Lambda(2s-1)}{\Lambda(2s)}y^{1-s} + 2\sqrt y\sum_{n\neq 0} \frac{|n|^{s-\frac 12}\sigma_{1-2s}(|n|)}{\Lambda(2s)} K_{s-\frac 12}(2\pi|n|y)e(nx),
\end{equation}
where $\sigma_a(n)$ is the sum of the $a$-th powers of the divisors of $n$, and $\Lambda(s) = \pi^{-s/2}\Gamma(s/2)\zeta(s)$.
The modified Eisenstein series $\Lambda(2s)F_0(z,s)$ is analytic in $\C\setminus\{0,1\}$ and is invariant under $s\mapsto 1-s$.
For $m\neq 0$ the Fourier expansion of $F_m(z,s)$ is given in Theorem~3.4 of \cite{fay} (see also Section~8 of \cite{DIT-geom}); for $m\neq 0$ and $\re(s)>1$ we have
\begin{equation} \label{eq:fm-fourier}
	F_m(z,s) = f_m(z,s) + \frac{2|m|^{1/2-s}\sigma_{2s-1}(|m|)}{(2s-1)\Lambda(2s)}y^{1-s} + 2\sqrt y \sum_{n\neq 0} \Phi(m,n;s) K_{s-\frac 12}(2\pi|n|y)e(nx),
\end{equation}
where
\begin{equation}
	\Phi(m,n;s) = \sum_{c>0} \frac{K(m,n,c)}{c} 
	\begin{cases}
		I_{2s-1}(4\pi \sqrt{|mn|} \, c^{-1}) & \text{ if } mn<0, \\
		J_{2s-1}(4\pi \sqrt{|mn|} \, c^{-1}) & \text{ if } mn>0 
	\end{cases}
\end{equation}
and $K(m,n,c)$ is the ordinary (weight 0) Kloosterman sum.
The following result is Proposition~3 of \cite{DIT-geom}.
\begin{proposition} \label{prop:residue}
	For any $m\neq 0$, the function $F_m(z,s)$ has a meromorphic continuation to $\re(s)>0$ with
	\begin{equation}
		\Res_{s=\frac 12+ir} (2s-1)F_m(z,s) = \sum_{\varphi\in \mathcal B_r} \langle \varphi,\varphi \rangle^{-1} 2a_\varphi(m) \varphi(z).
	\end{equation}
\end{proposition}

In Proposition~5 of \cite{DIT-geom} the authors show that the cycle integrals of $F_m(z,s)$ and $\partial_z F_m(z,s)$ over finite geodesics yield weighted sums of Kloosterman sums. 
The next proposition is a complementary result that evaluates the cycle integrals over infinite geodesics, provided that we make a small modification to the integrand as in \cite{and-inf}.
Suppose that $Q=[a,b,c]\in \mathcal Q_D$ with $D$ a square and let $\mathfrak a_1,\mathfrak a_2$ be the rational projective solutions to $ax^2+bxy+cy^2=0$.
For each $j=1,2$ there is a unique $\gamma_j\in \Gamma_\infty\backslash \Gamma$ such that $\gamma_j \mathfrak a_j=\infty$, and we define
\begin{equation}
	F_{m,Q}(z,s) = \sum_{\substack{\gamma\in \Gamma_\infty \backslash \Gamma \\ \gamma \neq \gamma_1,\gamma_2}} f_m(\gamma z,s).
\end{equation}
Since $F_{m,\sigma Q}(z,s) = F_{m,Q}(\sigma z,s)$ for all $\sigma\in \Gamma$, the integrals 
\begin{equation} \label{eq:s-integrals}
	\int_{C_Q} F_{m,Q}(z,s) y^{-1} |dz| \qquad \text{ and } \qquad \int_{C_Q} \partial_z F_{m,Q}(z,s) \, dz
\end{equation}
are well-defined, assuming they converge.
To show convergence, using Lemma~\ref{lem:square} we may assume that $Q=[0,|d|,c]$ and $0\leq c<|d|$ for $D=d^2$.
Then we can take $\mathfrak a_1 = \infty$ and $\gamma_1 = I$.
The Fourier expansions \eqref{eq:eis-fourier} and \eqref{eq:fm-fourier} show that for $\re(s)>1$ the integrals in \eqref{eq:s-integrals} converge at $\infty$, and the observation
\[
	F_{m,Q}(\gamma_2^{-1}z,s) = \sum_{\gamma\neq I,\gamma_2^{-1}}f_m(\gamma z,s)
\]
shows that the integrals converge at $\mathfrak a_2$.
For $\re(s)>1$ and $d$ a fundamental discriminant, define
\begin{equation}
	T_m(d) = \sum_{Q\in \Gamma\backslash \mathcal Q_{d^2}} \chi_d(Q) 
		\begin{dcases}
			\int_{C_Q} F_{m,Q}(z,s) y^{-1} |dz| & \text{ if }d>0, \\
			\int_{C_Q} i \partial_z F_{m,Q}(z,s) \, dz & \text{ if }d<0.
		\end{dcases}
\end{equation}
Then we have the following analogue of Proposition~5 of \cite{DIT-geom}.

\begin{proposition} \label{prop:integrals}
	Let $m\geq 0$ and $\re(s)>1$.
	Suppose that $d$ is a fundamental discriminant.
	Then
	\begin{equation}
	T_m(d) = 
	\begin{dcases}
		6\pi^{1/2}|d|^{3/2}m \sum_{n\mid m} n^{-3/2} \pfrac dn \Phi^+\left(d,\tfrac {m^2}{n^2}d;\tfrac {2s+1}4\right) & \text{ if }m>0, \\
		\frac{\Gamma(\frac s2+\frac{1-\sgn d}{2})^2|d|^s L(s,\chi_d)^2}{\Gamma(s)\zeta(2s)} & \text{ if }m=0,
	\end{dcases}
	\end{equation}
	where, for $p,q\equiv 0,1\pmod{4}$ and $pq>0$ we have
	\begin{equation}
		\Phi^+(p,q,s) = \frac{\Gamma(s-\frac{\sgn p}{4})\Gamma(s-\frac{\sgn q}4)}{3\sqrt{\pi}\, 2^{2-2s}\Gamma(2s-\frac 12)} (pq)^{-\frac 12} \sum_{4\mid c>0} \frac{K^+(p,q,c)}{c} J_{2s-1}\pfrac{4\pi\sqrt{pq}}{c}.
	\end{equation}
	Here $K^+(p,q,c)$ is the half-integral weight Kloosterman sum
	\begin{equation}
		K^+(p,q,c) = (1-i)\sum_{d\bmod c} \pfrac cd \ep_d \, e\left(\frac{p\bar d+qd}{c}\right) \times
		\begin{cases}
			1 & \text{ if $c/4$ is even}, \\
			2 & \text{ if $c/4$ is odd},
		\end{cases}
	\end{equation}
	with $\ep_d = 1$ if $d\equiv 1\pmod{4}$ and $\ep_d=i$ if $d\equiv 3\pmod{4}$.
\end{proposition}

\begin{proof}
	When $d,m>0$ this is (4.4) of \cite{and-inf} (see also Proposition~4 of that paper).
	For the case $d>0,m=0$, (4.4) of \cite{and-inf} reads
	\begin{equation}
		T_0(d) = \frac{\Gamma(\frac s2)^2}{4\Gamma(s)} d^s L(s,\chi_d) \sum_{c=1}^\infty \frac{K^+(d,0;4c)}{c^{s+1/2}}.
	\end{equation}
	By Lemma~4 of \cite{DIT-cycle} we have
	\begin{equation}
		T_0(d) = \frac{\Gamma(\frac s2)^2 d^s L(s,\chi_d)^2}{\Gamma(s)\zeta(2s)}.
	\end{equation}

	Now assume that $d<0$.
	We will closely follow the proof of Proposition~4 of \cite{and-inf}.
	From the proof of Lemma~5 of \cite{DIT-geom} (see (9.2) especially) we have
	\begin{equation}
		2i\partial_z F_{m,Q}(z,s) = \sum_{\substack{\gamma\in \Gamma_\infty \backslash \Gamma \\ \gamma \neq \gamma_1,\gamma_2}} f_{2,m}(\gamma z,s) \frac{d(\gamma z)}{dz},
	\end{equation}
	where $f_{2,m}(z,s) = \phi_{2,m}(y,s)e(mx)$ and
	\begin{equation}
		\phi_{2,m}(y,s) = 
		\begin{cases}
			 sy^{s-1} & \text{ if }m=0, \\
			 sm^{-1/2} (2\pi y)^{-1} \frac{\Gamma(s)}{\Gamma(2s)} M_{1,s-\frac 12}(4\pi my) & \text{ if }m > 0.
		\end{cases}
	\end{equation}
	We choose representatives $[c,|d|,0]$ for $\Gamma\backslash \mathcal Q_{d^2}$ as in Lemma~\ref{lem:square} so that $\chi_d(Q) = \pfrac dc$.
	Since $\pfrac dc=0$ when $\gcd(c,d)>1$ we can restrict the sum to those $c$ which are coprime to $d$.
	For $Q=[c,|d|,0]$ with $\gcd(c,d)=1$ we have
	\begin{equation}
		\mathfrak a_1 = (0:1), \gamma_1 = \pmatrix 0{-1}10, \quad \mathfrak a_2 = (d:c), \gamma_2 = \pmatrix ab{c}{|d|},
	\end{equation}
	for some $a,b\in \Z$.
	Thus
	\begin{equation}
		T_m(d) = \frac 12 \sum_{\substack{c\bmod |d|\\ \gcd(c,d)=1 \\Q=[c,|d|,0]}} \pfrac{d}{c} \sum_{\substack{\gamma\in \Gamma_\infty \backslash \Gamma \\ \gamma \neq \gamma_1,\gamma_2}} \int_{C_{\gamma Q}} e(mx) \phi_{2,m}(y,s) \, dz.
	\end{equation}
	The map $(\gamma,Q)\mapsto \gamma Q$ is a bijection from $\Gamma_\infty\backslash \Gamma \times \Gamma\backslash \mathcal Q_{d^2}$ to $\Gamma_\infty \backslash \mathcal Q_{d^2}$ which sends $(\Gamma_\infty\gamma_1,[c,|d|,0])$ to $[0,d,c+d\Z]$ and $(\Gamma_\infty\gamma_2,[c,|d|,0])$ to $[0,|d|,-b+d\Z]$.
	It follows that
	\begin{equation}
		T_m(d) = \frac 12 \sum_{\substack{Q\in \Gamma_\infty \backslash \mathcal Q_{d^2} \\ Q=[a,b,c],a\neq 0}} \chi_d(Q) \int_{C_Q} e(mx) \phi_{2,m}(y,s) \, dz.
	\end{equation}
	Since $\chi_d(-Q) = -\chi_d(Q)$ and the geodesic $C_{-Q}$ is the same set as $C_Q$ but with opposite orientation, we have
	\begin{equation}
		T_m(d) = \sum_{\substack{Q\in \Gamma_\infty \backslash \mathcal Q_{d^2} \\ Q=[a,b,c], a>0}} \chi_d(Q) \int_{C_Q} e(mx) \phi_{2,m}(y,s) \, dz.
	\end{equation}
	Each cycle $C_Q$ with $Q=[a,b,c]$ and $a>0$ can be parametrized by
	\begin{equation}
		z  = \re z_Q - e^{-i\theta} \im z_Q, \qquad 0\leq \theta \leq \pi,
	\end{equation}
	where
	\begin{equation}
		z_Q = -\frac{b}{2a} + i\frac{|d|}{2a}
	\end{equation}
	is the apex of the geodesic.
	Thus
	\begin{equation} \label{eq:phi-integral}
		\int_{C_Q} e(mx)\phi_{2,m}(y,s) \, dz = e\pfrac{-mb}{2a} H_m\pfrac{|d|}{2a},
	\end{equation}
	where
	\begin{equation}
		H_m(t) = it \int_0^\pi e(-mt\cos\theta) \phi_{2,m}(t\sin\theta,s)e^{-i\theta} \, d\theta.
	\end{equation}
	It follows that
	\begin{equation}
		T_m(d) = \sum_{a=1}^\infty H_m\pfrac{|d|}{2a} \sum_{\substack{b(2a) \\ b^2\equiv d^2(4a)}} \chi_d\left(\left[a,b,\tfrac{b^2-d^2}{4a}\right]\right) e\pfrac{-mb}{2a}.
	\end{equation}
	By Lemma~7 of \cite{DIT-geom} we have
	\begin{equation}
		H_m(t) = \frac{2\sqrt\pi \Gamma(\frac{s+1}{2}) t^{1/2}}{\Gamma(\frac s2)} J_{s-\frac 12} (2\pi|m|t)
	\end{equation}
	when $m\neq 0$, while when $m=0$ we have by \cite[(5.12.2)]{dlmf} that 
	\begin{equation}
		H_0(t) = 2\sqrt{\pi} t^{s} \frac{\Gamma(\frac{s+1}{2})}{\Gamma(\frac s2)}.
	\end{equation}
	The result follows after using Lemma~8 of \cite{DIT-geom}.
\end{proof}

We would like to apply Proposition~\ref{prop:residue} to the integrals appearing in Proposition~\ref{prop:integrals}, but the integrals in \eqref{eq:s-integrals} do not converge for $\re(s)=\frac 12$. 
However, the integrals
\begin{equation}
	\int_{C_Q} \left(F_{m,Q}(z,s) - c(s)F_{0,Q}(z,s) \right) y^{-1} |dz| \quad \text{ and } \quad \int_{C_Q} \partial_z\left( F_{m,Q}(z,s) - c(s)F_{0,Q}(z,s)\right) \, dz,
\end{equation}
where
\begin{equation}
	c(s) = \frac{2|m|^{1/2-s}\sigma_{2s-1}(|m|)}{(2s-1)\Lambda(2s-1)},
\end{equation}
do converge for $\re(s)>0$, as long as $s$ is not one of the poles of the integrands.
This is because the coefficient of $y^{1-s}$ equals zero in the Fourier expansion of $F_{m,Q}(z,s) - c(s)F_{0,Q}(z,s)$ at the cusps corresponding to the endpoints of $C_Q$.
Note that by Proposition~\ref{prop:residue} we have
\begin{align*}
	\Res_{s=\frac 12+ir} (2s-1)\left(F_{m,Q}(z,s) - c(s)F_{0,Q}(z,s)\right) 
	&= \Res_{s=\frac 12+ir} (2s-1) F_{m}(z,s) \\
	&= \sum_{\varphi\in \mathcal B_r} \langle \varphi,\varphi \rangle^{-1} 2a_\varphi(m)\varphi(z)
\end{align*}
because $c(s)F_{0,Q}(z,s)$ is analytic at $s=\frac 12+ir$, $r\neq 0$, and $f_m(\gamma_j z,s)$ is analytic for $s\in \C$.
So we get the next result by following the proof of Proposition~6 of \cite{DIT-geom} with only minor changes.

\begin{proposition} \label{prop:bd-bdp}
	For any even Hecke--Maass cusp form $\varphi\in \mathcal U_r$ there is a unique Hecke--Maass cusp form $\psi\in \mathcal V_r$ such that $\varphi$ is the Shimura lift of $\psi$ and such that for any fundamental discriminant $d$ we have
	\begin{equation}
		12\pi^{1/2}|d|^{\frac 32}|b_{\psi}(d)|^2 = \frac{1}{\langle \varphi,\varphi \rangle} \sum_{Q\in \Gamma\backslash\mathcal Q_{d^2}} \chi_d(Q)
		\begin{dcases}
			\int_{C_Q} \varphi(z) y^{-1}\, |dz| & \text{ if }d>0,\\
			\int_{C_Q} i\partial_z \varphi(z) \, dz & \text{ if }d<0.
		\end{dcases}
	\end{equation}
\end{proposition}

\begin{proof}[Proof of Theorem~\ref{thm:main}]
	Suppose that $d$ is a fundamental discriminant.
	Then by Lemma~\ref{lem:square} the quadratic forms $[0,|d|,c]$ with $0\leq c<|d|$ form a complete set of representatives for $\Gamma\backslash \mathcal Q_{d^2}$ and
	\begin{equation}
		\chi_d([0,|d|,c]) = \pfrac dc.
	\end{equation}

	Suppose first that $d>0$.
	If $\re(s)>1$ then
	\begin{align}
		\sum_{Q\in \Gamma\backslash \mathcal Q_D} \chi_d(Q) \int_{C_Q} \varphi(z) y^{s-1}\, |dz| 
		&= \sum_{c\bmod d} \pfrac dc \int_0^\infty \varphi(-\tfrac cd+iy) y^{s-1} dy \\
		&= 2\sum_{n\neq 0} a_\varphi(n) G(-n,d) \int_0^\infty y^{s-\frac 12}K_{ir}(2\pi|n|y) \, dy,
	\end{align}
	where $G(n,d)$ is the Gauss sum
	\begin{equation}
		G(n,d) = \sum_{c\bmod |d|} \chi_d(c) e\pfrac{nc}{|d|} = \chi_d(n) \sqrt{|d|} \times
		\begin{cases}
			1 & \text{ if }d>0, \\
			i & \text{ if }d<0.
		\end{cases}
	\end{equation}
	By \cite[(10.43.19)]{dlmf} we have
	\begin{equation} \label{eq:K-integral}
		\int_0^\infty y^{s-\frac 12}K_{ir}(2\pi|n|y) \, dy = \tfrac 14 (\pi|n|)^{-s-\frac 12} \Gamma(\tfrac s2+\tfrac{ir}{2} + \tfrac 14) \Gamma(\tfrac s2-\tfrac{ir}{2} + \tfrac 14).
	\end{equation}
	Thus, using that $a_\varphi(n) = a_\varphi(-n)$ we find that
	\begin{equation}
		\sum_{Q\in \Gamma\backslash \mathcal Q_D} \chi_d(Q) \int_{C_Q} \varphi(z) y^{s-1}\, |dz| = \pi^{-s-\frac 12} \sqrt{d} \, \Gamma(\tfrac s2+\tfrac {ir}2+\tfrac 14)\Gamma(\tfrac s2-\tfrac {ir}2+\tfrac 14) L(s+\tfrac 12,\varphi\times\chi_d).
	\end{equation}
	Setting $s=0$ and using Proposition~\ref{prop:bd-bdp} we get \eqref{eq:kohnen-zagier-formula}.

	Now suppose that $d<0$.
	A computation involving \cite[\S10.29, (10.30.2), and (10.40.2)]{dlmf} shows that
	\begin{equation}
		\partial_z \left[\sqrt y K_{ir}(2\pi|n|y)e(nx)\right] = \pi i n \sqrt y K_{ir}(2\pi |n|y)e(nx) + g(n,y)e(nx)
	\end{equation}
	for some function $g(n,y)$ which satisfies $g(-n,y)=g(n,y)$ and $g(n,y)\ll |n|^{1/2}e^{-2\pi|n|y}$ as $|n|y\to\infty$ and $g(n,y)\ll_n y^{-1/2}$ as $y\to 0$.
	So if $\re(s)>1$ we have
	\begin{equation}
		\sum_{Q\in \Gamma\backslash \mathcal Q_D} \chi_d(Q) \int_{C_Q} i\partial_z\varphi(z) y^{s}\, dz = -2\pi i\sum_{n\neq 0}n a_\varphi(n) G(n,d) \int_0^\infty y^{s+\frac 12}K_{ir}(2\pi |n|y) \, dy
	\end{equation}
	because $a_\varphi(-n)G(-n,d)g(-n,y) = -a_\varphi(n)G(n,d)g(n,y)$.
	Again using \eqref{eq:K-integral} we find that
	\begin{equation}
		\sum_{Q\in \Gamma\backslash \mathcal Q_D} \chi_d(Q) \int_{C_Q} i\partial_z\varphi(z) y^{s}\, dz = \pi^{-s-\frac 12} \sqrt{|d|} \, \Gamma(\tfrac s2+\tfrac {ir}2+\tfrac 34)\Gamma(\tfrac s2-\tfrac {ir}2+\tfrac 34) L(s+\tfrac 12,\varphi\times\chi_d).
	\end{equation}
	The result follows as in the previous case.
\end{proof}

\bibliographystyle{plain}
\bibliography{walds-bib}

\end{document}